\documentclass{article}
\usepackage{amsmath,amssymb,amsthm,amsfonts,color}
\newtheorem{thm}{Theorem}[section]

\newtheorem{lemma}[thm]{Lemma}
\newtheorem{prop}[thm]{Proposition}
\theoremstyle{definition}
\newtheorem{defn}[thm]{Definition}
\newtheorem{remark}[thm]{Remark}

\newcommand{\R}{\mathbb R}
\newcommand{\N}{\mathbb N}

\newcommand{\HL}{H_{0,L}^1(\mathcal C)}
\newcommand{\ve}{\varepsilon}
\title{Existence and multiplicity results for the fractional Laplacian in bounded domains}

\author{
{\sc Dimitri Mugnai} and 
{\sc Dayana Pagliardini}}

\date{}
\begin{document}

\maketitle
\begin{abstract}
In this paper, first we study existence results for a linearly perturbed elliptic problem driven by the fractional Laplacian. Then, we show a multiplicity result when the perturbation parameter is close to the eigenvalues. This latter result is obtained by exploiting the topological structure of the sublevels of the associated functional, which permits to apply a critical point theorem of mixed nature due to Marino and Saccon.
\end{abstract}


2000AMS Subject Classification: 35R11, 35J65, 35J20, 35J61.

\section{Introduction}

In this paper we consider the problem
\begin{equation}\label{P}
\begin{cases}
(-\Delta)^{1/2}u=\lambda u+g(x,u) & \mbox{in}\; \Omega\\
u=0 & \mbox{on}\; \partial\Omega,
\end{cases}\end{equation}
where $\Omega$ is a bounded domain of $\mathbb R^N$,  $N\geq 2$, $\lambda \in \mathbb R$ and $g:\Omega \times \mathbb R\rightarrow \mathbb R$ is a given function; more precise details will be given below.
Of course, problem \eqref{P} is a possible fractional counterpart of the problem
\begin{equation}\label{PC}
\begin{cases}
-\Delta u=\lambda u+g(x,u)  & \mbox{in}\; \Omega\\
 u=0  & \mbox{on}\; \partial \Omega,
\end{cases}
\end{equation}
which has been the object of extensive study in the last four decades, essentially with the aid of variational methods, when $N\geq 3$.

We will not go into details about existence and multiplicity results for \eqref{PC} according to different assumptions on $g$, since the bibliography would be huge, but we focus on multiplicity results with very general assumptions on $g$. The first result for $g(x,t)\sim |t|^{p-2}t$ was established by Ambrosetti and Rabinowitz in \cite{ar}, where the authors assumed that
\begin{itemize}
\item  $g:\overline{\Omega} \times \mathbb R \rightarrow \mathbb R$ is a continuous function; 
\item  there exist $a_1, a_2>0$ and $p\in \left(2,2N/(N-2)\right)$ such that for all $(x,t)\in \overline \Omega\times \mathbb R$,
\begin{equation}\label{crescg}
|g(x,t)|\le a_1+a_2|t|^{p-1};
\end{equation}
\item $g(x,t)=o(|t|)$ for $t \rightarrow 0$ uniformly in $\overline \Omega$;
\item there exists $R\geq0$ and $\mu>2$ such that for all $|t|>R$ and all $x\in \overline\Omega$ 
\begin{equation}\label{pgg}
0<\mu G(x,t)\le g(x,t)t,
\end{equation}
where $G(x,t)=\int_{0}^{t}{g(x,\sigma)d\sigma}.$
\end{itemize}
Let us note that, as a consequence of \eqref{pgg}, we get the existence of $c_1,c_2>0$ such that for all $(x,t)\in \overline \Omega\times \mathbb R$
\begin{equation}\label{Gmu}
G(x,t)\ge c_1|t|^\mu-c_2.
\end{equation}

Under this assumptions, in \cite{ar} it is proved that problem \eqref{PC} has two nontrivial solutions when $\lambda=0$, though the same proof holds if $\lambda<\lambda_1$, where $\lambda_1$ is the first eigenvalue of $-\Delta$ in $\Omega$ subject to homogeneous Dirichlet boundary conditions.

Almost twenty years later, in \cite{W}, assuming $g:\R\to\R$ is of class $C^1$, Wang proved that problem \eqref{PC} has three nontrivial solutions under the same related assumptions. Since then, thousands of papers have estabilished other multiplicity results weakening the assumptions on the superlinear and subcritical $g$ (see the recent paper by Mugnai and Papageorgiou \cite{MP} for more general operators). However, not much was done for the case $\lambda>\lambda_1$. The first result in proving  Wang's result for $\lambda>\lambda_1$ and close to an eigenvalue can be found in Mugnai~\cite{dm}, whose result is complemented by Rabinowitz--Su--Wang in \cite{RSW}. However, while in \cite{RSW} $g:\overline{\Omega}\times\R\to \R$ is assumed to be of class $C^1$, in \cite{dm}, for a superlinear and supercritical $g$, it is assumed that
\begin{itemize}
\item  $g:{\Omega} \times \mathbb R \rightarrow \mathbb R$ is a Carath\'eodory function; 
\item there exist $a_1, a_2>0$ and $p\in (2,2N/(N-2))$ such that \eqref{crescg} holds for all $t\in \mathbb R$ and for a.e. $x \in \Omega$;
\item $g(x,t)=o(|t|)$ for $t \rightarrow 0$ uniformly in $\Omega$;
\item the condition \eqref{pgg} holds for all $t \neq 0$ and for a.e. $x$ in $\Omega$ with $\mu=p$.
\end{itemize}
Let us remark that with these weak assumptions, inequality \eqref{pgg} does {\em not} imply \eqref{Gmu}, and for this one has to {\em assume} that
\begin{itemize}
\item there exists $c_1>0$ such that for all $t\in \mathbb R$ and for a.e. $x$ in $\Omega$, we have  $G(x,t)\ge c_1|t|^p$, see Mugnai \cite{adm}.
\end{itemize}

Going back to problem \eqref{P}, some remarks are needed. The operator $(-\Delta)^{1/2}$ which we consider is the spectral square root of the Laplacian, which should not be confused with the integro--differential operator defined, up to a constant, as
\[
-(-\Delta)^s u(x)=
\int_{\R^n}\frac{u(x+y)+u(x-y)-2u(x)}{|y|^{n+2s}}\,dy,
\,\,\,\,\, x\in \R^n.
\]
Indeed, in this case the homogeneous Dirichlet ``boundary conditions'' should be interpreted as $u\equiv 0$ in $\R^N\setminus\Omega$ (see \cite{f}, \cite{s0}, \cite{sv} and \cite{s}). In fact, in \cite{s0} the authors show that these two operators, though often denoted in the same way, are really different, with eigenvalues and eigenfunctions behaving quite differently.

As already said, we will consider the spectral square root of the Laplacian, defined according to the following procedure (see Cabr\'e and Tan \cite{ct} and Caffarelli and Silvestre \cite{cs}). Let $H^{1/2}(\Omega)$ denote the Sobolev space of order $1/2$, defined as
\[
H^{1/2}(\Omega):=\left\{ u\in L^2(\Omega)\,:\, \int_\Omega\int_\Omega \frac{|u(x)-u(y)|^2}{|x-y|^{N+1}}dxdy<\infty\right\}
\]
with norm
\[
\|u\|^2_{H^{1/2}(\Omega)}:=\int_\Omega u^2dx+\int_\Omega\int_\Omega \frac{|u(x)-u(y)|^2}{|x-y|^{N+1}}dxdy.
\]
Introducing the cylinder ${\mathcal C}:=\Omega\times (0,\infty)$ with lateral boundary $\partial_L {\cal C}=\partial \Omega \times [0,\infty)$, set

\[
H_{0,L}^1(\mathcal C):=\left\{v\in H^1({\mathcal C})\,:\,v=0 \mbox{ a.e. on }\partial_L{\cal C}\right\},
\]
and denote by $tr_\Omega$ the trace operator on $\Omega\times \{0\}$ for functions in $\HL$,
\[
tr_\Omega v(x,y):=v(x,0)\quad \mbox{ for all $(x,y)\in \mathcal C$}.
\]
Then, from standard results, we know that
\[
{\cal V}_0(\Omega):=\left\{u=tr_\Omega v\,:\,v\in \HL\right\}\subset H^{1/2}(\Omega),
\]
but a characterization of ${\cal V}_0$ is available from the following proposition.
\begin{prop}[\cite{ct}, Proposition 2.1]
We have
\[
\begin{aligned}
{\cal V}_0(\Omega)&=\left\{u\in H^{1/2}(\Omega)\,:\, \int_\Omega \frac{u^2}{{\rm d(x)}}dx<\infty\right\} \\
&=\left\{u\in L^2(\Omega)\,:\,u=\sum_{k=1}^\infty \alpha_k \varphi_k\ { \rm satisfies }\ \sum_{k=1}^\infty \alpha_k^2\lambda_k<\infty\right\}.
\end{aligned}
\]
Here $(\lambda_k^2,\varphi_k)_k$ is the Dirichlet spectral decomposition of $-\Delta$ in $\Omega$, $(\varphi_k)_k$ being an orthonormal basis in $L^2(\Omega)$, and ${\rm d}(x):={\rm dist}(x,\partial \Omega)$.
\end{prop}

From these preliminaries, by \cite[Proposition 2.2]{ct}, for $u=\sum_{k=1}^\infty \alpha_k \varphi_k$, we define
\begin{equation}\label{A1/2}
(-\Delta)^{1/2}u:=\sum_{k=1}^\infty \alpha_k \lambda_k\varphi_k.
\end{equation}

With this definition in hand, the purpose of this paper is to prove a multiplicity result for problem \eqref{P}, in a situation similar to that described above for \eqref{PC}, and, in view of the previous considerations, we assume that
\begin{description}
\item$(g_1)$  $g:{\Omega} \times \mathbb R \rightarrow \mathbb R$ is a Carath\'eodory function; 
\item$(g_2)$ there exist $a_1, a_2>0$ and $p\in \left(2,2N/(N-1)\right)$ such that \eqref{crescg} holds for all $t\in \mathbb R$ and for a.e. $x \in \Omega$;
\item$(g_3)$ $g(x,t)=o(|t|)$ for $t \rightarrow 0$ uniformly in $\Omega$;
\item$(g_4)$ we have $$0<p G(x,t)=p\int_{0}^{t}{g(x,\sigma)d\sigma}\le g(x,t)t$$ for all $t \neq 0$ and for a.e. $x\in \Omega$;
\item $(g_5)$ there exists $c_1>0$ such that $G(x,t)\ge c_1|t|^p$ for all $t\in \mathbb R$ and for a.e. $x\in \Omega$.
\end{description}

\begin{remark}\label{Osservazione 1}
Note that $(g_3)$ implies that $G(x,t) =o(|t|^2)$ as $t\rightarrow 0$ uniformly in $\Omega$ and from $(g_2)$ we get that $$|G(x,t)|\le a_1|t|+\dfrac{a_2}{p}|t|^p$$ for all $t\in \mathbb R$ and for a.e. $x\in \Omega$. As a consequence, $u=0$ solves problem \eqref{P}, and we look for nontrivial solutions.
\end{remark}

In the case $g(x,t)=|t|^{p-2}t$, $2<p<2N/(N-1)$, an existence result for $\lambda=0$ is proved in Cabr\'e and Tan \cite{ct}. The proof therein can be immediately extended to the case $\lambda<\lambda_1$, but we are not aware of existence results for $\lambda\geq \lambda_1$, nor for general nonlinearities $g$. For this, we state our first result:
\begin{thm}\label{Theorem 1}
If $(g_1)-(g_5)$ hold, then for every $\lambda\in \R$ problem \eqref{P} has one nontrivial solution.
\end{thm}

However, the previous result is standard, and we only present it for a complete description of the existence setting, as a counterpart of the result in Servadei and Valdinoci \cite{sv}, when the fractional Laplacian is represented by a nonlocal integral operator, already introduced in \cite{f,s}.

On the other hand, our main interest is providing a multiplicity theorem, which is much more involved. This result relies on the application of a critical point theorem of mixed type proved by Marino and Saccon in \cite{ms}, recalled in the Appendix, together with an additional linking theorem and a fine estimate of critical levels. More precisely, we prove the following multiplicity result in Wang's direction, the main result of this paper:
\begin{thm}\label{Theorem 2}
Assume $(g_1)-(g_5)$ hold. Then, for all $i\in \mathbb N$, $i\ge 2$, there exists $\delta_i>0$ such that problem \eqref{P} has at least three nontrivial solutions for all $\lambda \in (\lambda_i -\delta_i, \lambda_i)$.
\end{thm}

We conclude recalling that the $\nabla$--theorem we will employ has been extensively used in several contexts, in order to prove multiplicity results of different problems, such as elliptic problems of second and fourth order, variational inequalities and reversed variational inequalities, see, for instance, \cite{marinomugnai, MMS, MS, mug4, dm, bila, OL, fwang, fwang2, wzz}, and \cite{mbms}, where the analogous multiplicity result of this paper is considered when the underlying operator is the nonlocal one studied in \cite{f,sv,s}.

\section{Extended problem, preliminary lemmas and proof of Theorem $\ref{Theorem 1}$}\label{Disuguaglianze e lemmi tecnici}

We know from \cite{ct} that solving problem \eqref{P} is equivalent to solving its extension in the cylinder $\mathcal{C}=\Omega \times (0,\infty)$, that is,
\begin{equation}\label {E}
\begin{cases}
\Delta v=0 & \mbox{in}\; \mathcal{C},\\
v=0 & \mbox{on}\; \partial_L\mathcal{C},\\
\dfrac{\partial v}{\partial\nu}=\lambda v+g(x,v) & \mbox{in}\; \Omega \times \{0\},
\end{cases}
\end{equation}
where $\partial_L \mathcal{C}=\partial \Omega \times (0, \infty)$ is the cylinder lateral surface and
$\nu$ is the outer normal at $\mathcal{C}$ in $\Omega \times \{ 0 \}$. Let us briefly recall the relation between \eqref{E} and \eqref{P}.

First, we will look for weak solutions to \eqref{E}. For this, we note that the Sobolev space
$H_{0, L}^1 (\mathcal{C})$ introduced above is an Hilbert space when endowed with the norm
\[
\|v\|=\left( \int_{\mathcal{C}}{|Dv|^2\; dxdy}\right)^{1/2},
\]
induced by the inner product
\[
\langle v,w\rangle = \int_{\mathcal C}Dv\cdot Dw\,dxdy.
\]

Following the  ``Dirichlet to Neumann'' approach in a bounded domain $\Omega$ of $\mathbb R^N$ (cfr. \cite{ct}), for a given $u\in {\cal V}_0(\Omega)$ we consider the harmonic extension $v$ of $u$, i.e. the solution of the problem
\[
\begin{cases}
\Delta v=0 & \mbox{in}\; \mathcal{C},\\
v=0 & \mbox{on}\; \partial_L\mathcal{C},\\
v=u & \mbox{in}\; \Omega \times \{0\},
\end{cases}
\]
which is well defined by \cite[Lemma 2.8]{ct}. Now, by \cite[Proposition 2.2]{ct}, we can give a definition of  $(-\Delta)^{1/2}:{\cal V}_0(\Omega)\to {\cal V}_0(\Omega)^*$, equivalent to \eqref{A1/2}, as the operator that maps the Dirichlet datum  $u$ in the Neumann value of its harmonic extension
\[
(-\Delta)^{1/2}u=\frac{\partial v}{\partial \nu}(\cdot,0),
\]
and hence, if $v$ solves \eqref{E}, then $u=tr_{\Omega}v$ solves \eqref{P}, see \cite[Proposition 2.2]{ct}.
For this reason, from now on, we will look for weak solutions to \eqref{E}.

Of course, problem \eqref{E} is variational and its solutions are critical points of the $C^1$ functional $f_{\lambda}:H_{0, L}^1(\mathcal{C})\rightarrow \mathbb R$ defined by
$$
f_{\lambda}(u)=\frac{1}{2}\int_{\mathcal{C}} {|Du|^2 dx dy}-\frac{\lambda}{2}\int_{\Omega}{u^2 dx}-\int_{\Omega}{G(x,u) dx}.
$$ 

Before attacking functional $f_\lambda$, we recall some other tools we will use in the following.  
Recalling that $({\lambda_k, \varphi_k})_k $ denote the eigenvalues and the associated eigenfunctions of $-\Delta$ with homogeneous Dirichlet condition on $\partial \Omega$ with  $({\varphi_k})_k$ an orthonormal basis of $L^2(\Omega)$, from \cite{ct}, we have
\[
H_{0, L}^1(\mathcal{C})=\Big\{v(x,y)=\sum_{k=1}^{\infty}{b_k \varphi_k(x) e^{-\lambda_k y}} \quad (x,y)\in \mathcal{C}\,: \, \sum_{k=1}^{\infty}\lambda_k b_k^2 <\infty \Big\}.
\]

Setting $e_k(x,y)=\varphi_k(x)e^{-\lambda_k y}$, we observe that the $e_k$ are orthogonal in $H_{0, L}^1(\mathcal{C})$. Now, for every integer $i\geq 1$, put
\[
H_i=\mbox{span}(e_1, \cdots, e_i) \ \mbox{ and }\ H_i^{\bot}=\overline{\mbox{span}(e_{i+1},\cdots)}.
\]
Then, a simple calculation gives the proof of the following inequalities:
\begin{prop}\label{PD}
If  $u\in H_i$, then
\begin{equation}\label{1}
\int_{\mathcal{C}}{|Du|^2\; dx dy}\le \lambda_{i}\int_{\Omega}{u^2\; dx}.
\end{equation}\end{prop}
\begin{prop}[Constrained Poincar\'e inequality]\label{PDV}
If $u\in H_i^{\bot}$, then
\begin{equation}\label{2}
\int_{\mathcal{C}}{|Du|^2\; dx dy}\ge \lambda_{i+1}\int_{\Omega}{u^2\; dx}.
\end{equation}
\end{prop}

We will also use the continuous inclusions (see \cite[Lemma 2.4]{ct})
\begin{equation}\label{cont}
H_{0, L}^1(\mathcal{C}) \hookrightarrow L^r(\Omega) \quad \mbox{for all $r \in \left[1,\frac{2N}{N-1}\right]$},
\end{equation}
and the compact ones (see \cite[Lemma 2.5]{ct})
\begin{equation}\label{imcomp}
H_{0, L}^1(\mathcal{C}) \hookrightarrow L^r(\Omega) \quad \mbox{for all $r \in \left[1,\frac{2N}{N-1}\right)$}.
\end{equation}

Now, we have all the ingredients to look for critical points of functional $f_\lambda$, i.e. to solve problem \eqref{E}. As usual, the first step in applying variational methods is the following result:
\begin{prop}
If $c\in\R$, then $f_\lambda$ satisfies the $(PS)_c$ condition, namely: every sequence $(u_n)_n$ such that $f_\lambda(u_n)\to c$ and $f_\lambda'(u_n)\to 0$ as $n\to \infty$ has a converging subsequence.
\end{prop}
\begin{proof}
Let $(u_n)_n\subset H_{0, L}^1(\mathcal{C})$ be a Palais--Smale sequence at level $c\in\R$. Then, taken $k\in (2,p)$, there exis $M,N>0$ such that
\begin{equation}\label{psc1}
kf_\lambda(u_n)-f_\lambda'(u_n)u_n\leq M+N\|u_n\| \mbox{ for all }n\in \N.
\end{equation}
On the other hand, by $(g_4)$ and $(g_5)$, we get
\begin{equation}\label{psc2}
\begin{aligned}
&kf_\lambda(u_n)-f_\lambda'(u_n)u_n\\
&=\left(\frac{k}{2}-1\right)\|u_n\|^2-\lambda\left(\frac{k}{2}-1\right) \int_\Omega u_n^2dx+\int_\Omega[g(x,u_n)u_n-kG(x,u_n)]dx\\
&\geq \left(\frac{k}{2}-1\right)\|u_n\|^2-\lambda\left(\frac{k}{2}-1\right) \int_\Omega u_n^2dx+(p-k)\int_\Omega G(x,u_n)\,dx\\
& \geq \left(\frac{k}{2}-1\right)\|u_n\|^2-\lambda\left(\frac{k}{2}-1\right) \int_\Omega u_n^2dx+(p-k)c_1\int_\Omega |u_n|^pdx.
\end{aligned}
\end{equation}
By Young's inequality, for every $\ve>0$ there exists $D_\ve>0$ such that
\[
u_n^2\leq \ve |u_n|^p+D_\ve.
\]
Hence, \eqref{psc2} implies that
\[
kf_\lambda(u_n)-f_\lambda'(u_n)u_n\geq \left(\frac{k}{2}-1\right)\|u_n\|^2+\left[(p-k)c_1-\ve |\lambda| \left(\frac{k}{2}-1\right)\right]\int_\Omega |u_n|^pdx-D_\ve.
\]
Choosing $\ve$ sufficiently small, we finally get
\[
kf_\lambda(u_n)-f_\lambda'(u_n)u_n\geq \left(\frac{k}{2}-1\right)\|u_n\|^2-D_\ve,
\]
and by \eqref{psc1}, we get that $(u_n)_n$ is bounded. Then, we can assume that $u_n\rightharpoonup u$ in $H_{0, L}^1(\mathcal{C})$ and $u_n\to u$ in $L^r(\Omega)$ for all $r\in \left[1,2N/(N-1)\right)$.

Since $f_\lambda'(u_n)(u_n-u)\to 0$ as $n\to \infty$, we immediately get that that $u_n$ converges strongly to $u$, i.e. $(PS)_c$ holds.
\end{proof}

At this point we can prove Theorem \ref{Theorem 1}:
\begin{proof}[Proof of Theorem $1.3$]
First of all we observe that, from the Remark \ref{Osservazione 1}, $f_{\lambda}(0)=0$ and by $(g_2)$ and $(g_3)$, we get that, given $\epsilon>0$, there exists $C_\epsilon>0$ such that 
\begin{equation}\label{G}
G(x,s)\le \epsilon |u|^2+C_\epsilon |s|^p
\end{equation}
for a.e. $x\in \Omega$ and all $s\in \R$.

Now, we have to distinguish two cases: $\lambda<\lambda_1$ and $\lambda \in [\lambda_i,\lambda_{i+1})$, for some $i\in\mathbb N$.\\
\textit{First case: $\lambda<\lambda_1$.} 
We want to apply the mountain pass theorem (see \cite{ar}). Supposing $\lambda>0$ (the other case being easier), by \eqref{G}, the Poincar\'e inequality \eqref{1} for $i=0$ and by the continuous embedding of $H^1_{0,L}(\mathcal C)$ in $L^2(\Omega)$ and in $L^p(\Omega)$, we have that
\[
\begin{aligned}
f_{\lambda}(u)&\ge \dfrac{1}{2}\|u\|^{2}-\dfrac{\lambda}{2\lambda_1}\|u\|^2-\dfrac{\epsilon}{2} \|u\|_{L^2(\Omega)}^2-C_\epsilon \|u\|_{L^p(\Omega)}^p\\
&\ge \dfrac{1}{2}\Big(1-\dfrac{\lambda}{\lambda_1}-\dfrac{c\epsilon}{\lambda_1}\Big)\|u\|^2-C\|u\|^p,
\end{aligned}
\]
for some absolute constants $c,C>0$. Now, choosing $\epsilon$ enough small, we can suppose 
$A=\dfrac{1}{2}\Big(1-\dfrac{\lambda}{\lambda_1}-\dfrac{c\epsilon}{\lambda_1}\Big)>0$ and taking $\rho$ sufficiently small we have that 
\[
\inf_{\|u\|=\rho} f_{\lambda}(u)\ge \alpha>0
\]
for some $\alpha>0$. Moreover, if $u\neq 0$ and $t>0$, by $(g_5)$ we have
\[
\begin{aligned}
f_{\lambda}(tu)&=\dfrac{t^2}{2}\|u\|^2-\dfrac{\lambda}{2}t^2\|u\|_2^2-\int_\Omega{G(x, tu)\; dx}\\
&\le \dfrac{t^2}{2}\|u\|^2-\dfrac{\lambda}{2}t^2\|u\|_2^2-c_1t^p\|u\|_p^p\rightarrow -\infty,
\end{aligned}
\]
since $p>2$. Recalling that the $(PS)_c$-condition holds  for all $c\in \mathbb R$, the Mountain Pass Theorem implies that \eqref{P} has a nontrivial solution.

\textit{Second case: $\lambda \in [\lambda_i,\lambda_{i+1})$ for some $i\in \N$.}  
In this case we want to apply the linking theorem (see the Appendix).
If $u\in H_i$, from $\eqref{1}$ and $(g_4)$, we have
\[
f_{\lambda}(u)\le \Big(\dfrac{\lambda_i}{2}-\dfrac{\lambda}{2}\Big)\int_\Omega u^2dx-\int_{\Omega}{G(x,u)\;dx}\le 0.
\]
Moreover, by \eqref{G} and \eqref{2}, if $u\in H_i^\bot$, then we have
\[
\begin{aligned}
f_{\lambda}(u)&\ge \dfrac{1}{2}\|u\|^{2}-\dfrac{\lambda}{2\lambda_{i+1}}\|u\|^2-\int_\Omega{G(x,u)\;dx}\\
&\ge \dfrac{1}{2}\Big(1-\dfrac{\lambda}{\lambda_{i+1}}\Big)\|u\|^2-\dfrac{\epsilon}{2}\int_\Omega u^2dx-C_\epsilon \int_\Omega u^pdx>0.
\end{aligned}
\]
Thus, by \eqref{cont}, we obtain
\[
f_{\lambda}(u)\geq  \dfrac{1}{2}\Big(1-\dfrac{\lambda}{\lambda_{i+1}}-\dfrac{c\epsilon}{2}\Big)\|u\|^2-C\|u\|^p
\]
for some $c,C>0$ and for all $u\in H_i^\bot$. Choosing $\epsilon$ and $\rho>0$ small enough, we get
\[
\inf_{u\in H_i^\perp,\above 0pt  \|u\|=\rho} f_{\lambda}(u)\ge \alpha>0.
\]
Finally, if $u\in H_i$ and $t>0$, by $(g_5)$ we have
\[
\begin{aligned}
f_{\lambda}(u+te_{i+1})&=\dfrac{1}{2}\|u+te_{i+1}\|^2-\dfrac{\lambda}{2}\int_\Omega (u+te_{i+1})^2dx-\int_{\Omega}{G(x,u+te_{i+1})\;dx}\\
&\le \dfrac{1}{2}\|u+te_{i+1}\|^2-\dfrac{\lambda}{2}\int_\Omega (u+te_{i+1})^2dx-c_1\int_\Omega|u+te_{i+1}|^pdx\rightarrow -\infty
\end{aligned}
\]
as $t\to \infty$, since all norms are equivalent in $H_i\oplus (e_{i+1})$.

Recalling that $(PS)_c$-condition holds for all $c\in \mathbb R$, the Linking Theorem implies that problem \eqref{P} has a nontrivial solution.
\end{proof}
 
\section{Proof of Theorem $\ref{Theorem 2}$}\label{Principal Theorem}
 
In this section we prove Theorem \ref{Theorem 2}, the main contribution of this paper. For this, from now on, we assume that there exist $i\leq j$ in $\N$ such that $\lambda_{i-1}<\lambda_{i}=\cdots=\lambda_{j}<\lambda_{j+1}$.

Let us start by introducing some notations. If $i<j$ in $\mathbb N$, we introduce the following sets:
$$S_{j}^{+}(\rho)=\{u\in H_{j}^{\bot} :\; \|u\|=\rho \},$$ and 
$$ T_{i,j}(R)=\Big\{u\in H_{j}\; : \; \|u\|=R\Big\}\cup \Big\{u\in H_{i} :\; \|u\|\le R\Big\}.$$

We can now state our first Lemma.
\begin{lemma}\label{Lemma 1}
Assume $i, j\ge 2$ are such that $\lambda_{i-1}<\lambda_{i}=\cdots=\lambda_{j}<\lambda_{j+1}$ and $\lambda \in (\lambda_{i-1}, \lambda_{j})$. Then there exist $R$ and $\rho$ with $R>\rho>0$ such that $$\sup f_{\lambda}(T_{i-1,j}(R))<\,\inf f_{\lambda}(S_{i-1}^{+}(\rho)).$$
\end{lemma}
\begin{proof}
From \eqref{G}, \eqref{2} and by the compact embedding in $L^2(\Omega)$ and in $L^p(\Omega)$ given in \eqref{imcomp}, we have the existence of $\rho>0$ such that
$$
\mbox{inf}\;f_{\lambda}(S_{i-1}^{+}(\rho))>0.
$$

Moreover, it is clear that  $f_{\lambda}(H_{i-1})\le 0$. We conclude the proof by showing that 
$$
\lim_{\substack{\|u\| \to \infty \\ u\in H_{j}}} f_{\lambda}(u)=-\infty.
$$
Such a result easily follows from $(g_5)$ and from inequality \eqref{1}. Indeed, if $u\in H_{j}$, then
$$f_{\lambda}(u)\le \dfrac{1}{2}\|u\|^{2}-\dfrac{\lambda}{2\lambda_{j}}\|u\|^{2}-c_1\|u\|_p^p,$$ and since all norms in $H_j$ are equivalent, the lemma follows.
\end{proof}

Now take
\[
a\in \Big(\mbox{sup}\;f_{\lambda}(T_{i-1, j}(R), \mbox{inf}\;f_{\lambda}(S_{i-1}^{+}(\rho))\Big)
\]
and
$b>\mbox{sup}\;f_{\lambda} (\overline{B_{j}(R)})$,
where $B_{j}(R)$ is the ball in $H_{j}$ with radius $R$. Then, we have the following lemma.

\begin{lemma}\label{Lemma 2}
Suppose there exist integers $i, j \ge 2$ such that $\lambda_{i-1}<\lambda_{i}=\cdots=\lambda_{j}<\lambda_{j+1}$.Then, for every $\delta >0$ there exists $\epsilon_{0}>0$ such that $\forall \; \lambda \in [\lambda_{i-1}+\delta, \lambda_{j+1}-\delta]$, the only critical point $u$ of $f_{\lambda}$ in $H_{i-1}\oplus H_{j}^{\bot}$ such that $f_{\lambda}(u)\in [-\epsilon_{0}, \epsilon_{0}]$, is the trivial one.
\end{lemma}
\begin{proof}
Assume by contradiction that there exist $\delta >\nolinebreak0$, $\lambda_{n} \in [\lambda_{i-1}+\delta, \lambda_{j+1}-\delta]$ and $(u_{n})_n$ in $H_{i-1}\oplus H_{j}^{\bot}\backslash \{0\}$ such that
$$
f_{\lambda_{n}}(u_{n})\longrightarrow 0 \quad \mbox{ if }n\to \infty
$$
and such that for all $ z\in H_{i-1} \oplus H_{j}^{\bot}$, we have
\begin{equation}\label{3}
\int_{\mathcal{C}}{Du_{n}\cdot Dz\; dx dy}-\lambda_{n}\int_{\Omega}{u_{n} z\;dx}-\int_{\Omega}{g(x,u_{n})z\; dx}=0,
\end{equation}
where
\[
f_{\lambda_{n}}(u_{n})=\dfrac{1}{2} \int_{\mathcal{C}}{|Du_{n}|^2\; dx dy}-\dfrac{\lambda_{n}}{2}\int_{\Omega}{u_{n}^2\;dx}-\int_{\Omega}{G(x, u_{n})\;dx} .
\]
Up to a subsequence, we can assume that $\lambda_{n}\rightarrow \lambda$ in $[\lambda_{i-1}+\delta, \lambda_{j+1}-\delta]$. Choose $z=u_{n}$ in \eqref{3}. Then, by $(g_4)$ we obtain
$$\begin{aligned} 
0&=\int_{\mathcal{C}}{|Du_{n}|^2\; dx dy}-\lambda_{n}\int_{\Omega}{u_{n}^2\;dx}-\int_{\Omega}{g(x,u_{n})u_{n}\;dx}\\
&=2f_{\lambda_{n}}(u_{n})+\int_{\Omega}{[2G(x,u_{n})-g(x,u_{n})u_{n}]\;dx}\\ 
&\le2f_{\lambda_{n}}(u_{n})+(2-p)\int_{\Omega}{G(x,u_{n})\;dx}.
\end{aligned}$$
In particular, we deduce that
\begin{equation}\label{4}
\lim_{n \to \infty}{\int_{\Omega}{G(x,u_n)\; dx}}=0.
\end{equation}

Now, let $v_n$ in $H_{i-1}$ and $w_n$ in $H_{j}^{\bot}$ be such that $u_n=v_n+w_n$ for every $n\in \mathbb N$, and choose $z=v_n-w_n$ in \eqref{3}. Then, we have
$$\begin{aligned}
&\int_{\mathcal{C}}{|Dv_n|^2\; dx dy}-\lambda_n\int_{\Omega}{v_n^2\; dx}-\left(\int_{\mathcal{C}}{|Dw_n|^2\; dx dy}-\lambda_n\int_{\Omega}{w_n^2\;dx}\right)\\
&=\int_{\Omega}{g(x,u_n)(w_n-v_n)\; dx}\quad \forall \; n\in \mathbb N.
\end{aligned}$$
By \eqref{1} and \eqref{2}, we get the existence of $c=\delta/\lambda_{j+1}>0$ independent of $n$, $(c=\delta /{\lambda_{j+1}}$), such that 
\begin{equation}\label{5}
c\|u_n\|^2\le \int_{\Omega}{g(x,u_n)(v_n-w_n)\;dx}\quad \forall\; n\in \mathbb N.
\end{equation}
Here we used the fact that $v_n$ and $w_n$ are orthogonal, so that $\|u_n\|^2=\|v_n\|^2+\|w_n\|^2$ for every $n\in \N$.

Moreover, by H\"older's inequality, we get
\begin{equation}\label{g}\begin{aligned}
\left|\int_{\Omega}{g(x,u_n)(v_n-w_n)\; dx}\right| \le &\left(\int_{\Omega}{|g(x,u_n)|^{p/(p-1)}\; dx}\right)^{1-1/p}\\
&\cdot \left(\int_{\Omega}{|v_n-w_n|^p\; dx}\right)^{1/p}.
\end{aligned}\end{equation}
By the Sobolev embedding and by the compact embedding in \eqref{imcomp}, there exists an universal constant $\gamma_p>0$ such that
\begin{equation}\label{sopra}
\|v_n-w_n\|_{L^p(\Omega)}\le \gamma_p \|v_n-w_n\|=\gamma_p \|u_n\|.
\end{equation}
In this way, since $u_n \neq 0$, \eqref{5}, \eqref{g} and \eqref{sopra} imply that there exists $c'>0$ such that
\begin{equation}\label{6}
\|u_n\|\le c'\left(\int_{\Omega}{|g(x, u_n)|^{p/(p-1)}\; dx}\right)^{(p-1)/p} \quad \forall \;n\in \mathbb N.
\end{equation}

Up to a subsequence, there are two possibilities: either $\|u_n\| \to \infty$, or $\|u_n\|$ is bounded.

\textit{First case: $\|u_n\| \to \infty$.} Without loss of generality, we can suppose that there exists $u \in H_{i-1} \oplus H_j^{\bot}$ such that $u_n/\|u_n\| \rightharpoonup u$ in $\HL$ and $u_n/\|u_n\|\to u$ in $L^r(\Omega)$ for every $r\in \left[1,2N/(N-1)\right)$, see \eqref{imcomp}. 
First of all, \eqref{4}  implies that 
$$
0\leftarrow 2\dfrac{f_{\lambda_n}(u_n)}{\|u_n\|^2}\longrightarrow 1-\lambda \int_{\Omega}{u^2\;dx} \mbox{ as }n\to \infty,
$$
so that $u\not\equiv 0$. Moreover $(g_2)$ and $(g_5)$ imply that
$$\int_{\Omega}{|g(x, u_n)|^{p/(p-1)}\; dx}\le a_1'+a_2'\int_{\Omega}{G(x, u_n)\; dx}.$$
But the last quantity is bounded by \eqref{4}, while \eqref{6} leads to a contradiction.

\textit{Second case: $\|u_n\|$ is bounded.} As before, we can suppose that there exists $u \in H_{i-1} \oplus H_j^{\bot}$ such that $u_n \rightharpoonup u$ in $H_{i-1}\oplus H_j^\bot$ and $u_n \to u$ in $L^r(\Omega)$ for every $r\in \left[1,2N/(N-1)\right)$.
Moreover \eqref{4} and $(g_5)$ imply that $u=0$.

If $u_n \rightarrow 0$, then by \eqref{6} and $(g_3)$, we would have $$1\le \lim_{n\to \infty}{c'\dfrac{\left(\displaystyle\int_{\Omega}{|g(x, u_n)|^{p/(p-1)}\; dx}\right)^{(p-1)/p}}{\|u_n\|}}=0,$$
which is absurd. So, there should exist $\sigma >0$ such that $\|u_n\| \ge \sigma$ for all $n\in \mathbb N$; but also in this case, since $u_n \rightarrow 0$ in $L^p(\Omega)$, from \eqref{6} we would obtain 
$$\sigma \le \lim_{n\to \infty}{c'\left(\int_{\Omega}{|g(x, u_n)|^{p/(p-1)}\; dx}\right)^{(p-1)/p}}=0,$$
which is also absurd.
\end{proof}

For the following result we denote by $P: H_{0, L}^1(\mathcal{C})\longrightarrow \mbox{span}(e_i,\cdots, e_j)$ and $Q:H_{0, L}^1(\mathcal{C})\longrightarrow H_{i-1}\oplus H_i^\bot$ the orthogonal projections.
\begin{lemma}\label{Lemma 3}
Suppose there exist integers $i, j \ge 2$ such that $\lambda_{i-1}<\lambda_{i}=\cdots=\lambda_{j}<\lambda_{j+1}$. Let $\lambda \in \mathbb R$ and $(u_n)_n$ in $H_{0, L}^1(\mathcal{C})$ be such that $(f_\lambda(u_n))_n$ is bounded, $Pu_n\rightarrow 0$ and $Q\nabla f_{\lambda}\rightarrow 0$ as $n\to \infty$. Then $(u_n)_n$ is bounded.
\end{lemma}
\begin{proof}
Assume by contradiction that $(u_n)_n$ in unbounded. Then we can suppose that there exists $u$ in $H_{0, L}^1(\mathcal{C})$ such that $u_n/\|u_n\|\rightharpoonup u$ in $H_{0, L}^1(\mathcal{C})$ and $u_n/\|u_n\|\to u$ in $L^r(\Omega)$ for every $r\in \left[1,2N/(N-1)\right)$ as $n\to \infty$.

Note that $u_n=Pu_n+Qu_n$, $Pu_n\rightarrow 0$ and $Q\nabla f_{\lambda}(u_n)\rightarrow 0$, where $\nabla f_{\lambda}(u_n)=V_n$ is such that $$f_{\lambda}'(u_n)v=\int_{\mathcal{C}}{V_n \cdot Dv\;dx dy} \quad \forall \; v\in H_{0, L}^1(\mathcal{C}).$$ 
So we obtain $$V_n=u_n-K(\lambda u_n+g(x, u_n)),$$ where $K: L^2(\Omega)\longrightarrow H_{0, L}^1(\mathcal{C})$ is the operator defined by $K(u)=v$, where $v$ is the weak solution $v$ of the problem
\begin{equation}\label {K}
\begin{cases}
\Delta v=0 & \mbox{in}\; \mathcal{C},\\
v=0 & \mbox{on}\; \partial_L\mathcal{C},\\
\dfrac{\partial v}{\partial\nu}=u & \mbox{in}\; \Omega \times \{0\}.
\end{cases}
\end{equation}

In particular, we have
\[
\begin{aligned}
&\langle Q\nabla f_{\lambda}(u_n),u_n\rangle=\langle \nabla f_{\lambda}(u_n),u_n\rangle-\langle P\nabla f_{\lambda}(u_n),u_n\rangle \\
&=\|u_n\|^2-\lambda \int_{\Omega}{u_n^2\; dx}-\int_{\Omega}{g(x,u_n)u_n\;dx}\\
&-\int_{\mathcal{C}}{D(P(u_n-K(\lambda u_n+g(x,u_n))))\cdot Du_n\; dx dy}.
\end{aligned}
\]
But for every $z\in H_{0, L}^1(\mathcal{C})$, $Pz$ is a smooth function and $DPu_n=PDu_n$, because $u\in$ span$(e_i,\cdots,e_j)$ and $Pz \bot Qz$. In this way the last integral in the previous equation is equal to
$$\int_{\mathcal{C}}{|DPu_n|^2\; dx dy}-\lambda \int_{\Omega}{|Pu_n|^2\; dx}-\int_{\Omega}{g(x,u_n)Pu_n\;dx}.$$
As a consequence, we have
\begin{equation}\label{7}\begin{aligned}
&\langle Q\nabla f_{\lambda}(u_n),u_n\rangle=2f_{\lambda}(u_n)+2\int_{\Omega}{G(x,u_n)\; dx}\\
&-\int_{\Omega}{g(x,u_n)u_n\; dx}\\
&-\|Pu_n\|^2+\lambda \int_{\Omega}{|Pu_n|^2\; dx}+\int_{\Omega}{g(x, u_n)Pu_n\; dx}.
\end{aligned}\end{equation}

Now, observe that $(g_2)$ implies that $$\lim_{n\to \infty}{\dfrac{\displaystyle \int_{\Omega}{|g(x,u_n)Pu_n|\;dx}}{\|u_n\|^{p-1}}=0},$$
since $$|g(x,u_n)Pu_n|\le \|Pu_n\|_{\infty}(a_1+a_2|u_n|^{p-1})$$ and $\|Pu_n\|_{\infty}\rightarrow 0$ as $n\to \infty$ by assumption.
So, starting from \eqref{7}, using $(g_4)$ and dividing by $\|u_n\|^{p-1}$ for $p-1>1$, we get $$\lim_{n\to \infty}{\dfrac{\displaystyle \int_{\Omega}{G(x,u_n)\;dx}}{\|u_n\|^{p-1}}=0}.$$
Moreover, $(g_5)$ implies that $$\lim_{n\to \infty}{\dfrac{\displaystyle \int_{\Omega}{|u_n|^p\;dx}}{\|u_n\|^{p-1}}=0},$$ and so $u \equiv 0$.
Now, dividing $2f_{\lambda}(u_n)$ by $\|u_n\|^2$, we have
\begin{equation}\label{8}
\lim_{n\to \infty}{\dfrac{\displaystyle\int_{\Omega}{G(x,u_n)\;dx}}{\|u_n\|^2}}=\dfrac{1}{2}
\end{equation}
and so, by $(g_5)$, there exists a constant $c>0$ such that
\begin{equation}\label{9}
\int_{\Omega}{|u_n|^p\; dx}\le c\|u_n\|^2.
\end{equation}
Now, let us show that
\begin{equation}\label{10}
\lim_{n\to \infty}{\dfrac{\displaystyle\int_{\Omega}{|g(x,u_n)Pu_n|\;dx}}{\|u_n\|^2}}=0.
\end{equation}
Indeed, $(g_2)$ and H\"{o}lder's inequality imply that
\[
\begin{aligned}
&\lim_{n\to \infty}{\dfrac{\displaystyle\int_{\Omega}{|g(x,u_n)Pu_n|\;dx}}{\|u_n\|^2}}\\& \le \dfrac{\|Pu_n\|_{\infty}}{\|u_n\|^2}\left(a_1+a_2\int_{\Omega}{|u_n|^{p-1}\;dx}\right)\\
&\le\|Pu_n\|_{\infty}\left[\dfrac{a_1}{\|u_n\|^2}+\dfrac{a_2'}{\|u_n\|^{2/p}}\left(\dfrac{\int_{\Omega}{|u_n|^p\; dx}}{\|u_n\|^2}\right)^{(p-1)/p}\right],
\end{aligned}
\]
and  equality \eqref{10} follows from \eqref{9} and the fact that $p>2$.

In this way \eqref{7}, $(g_4)$ and \eqref{10} imply that
$$\lim_{n\to \infty}{\dfrac{\displaystyle\int_{\Omega}{G(x,u_n)\;dx}}{\|u_n\|^2}}=0,$$
which contradicts \eqref{8}.
\end{proof}

Now, by Lemma~\ref{Lemma 2} and Lemma~\ref{Lemma 3}, we can prove the following result:
\begin{prop}\label{prop 1}
Suppose there exist integers $i, j \ge 2$ such that $\lambda_{i-1}<\lambda_{i}=\cdots=\lambda_{j}<\lambda_{j+1}$. Then, for all $\delta >0$ there exists $\epsilon_0>0$ such that $\forall$ $\lambda \in [\lambda_{i-1}+\delta, \lambda_{j+1}-\delta]$ and $\forall$ $\epsilon', \epsilon''\in (0,\epsilon_0)$ with  $\epsilon'<\epsilon''$, condition $(\nabla)(f_\lambda,H_{i-1}\oplus H_j^\bot, \epsilon',\epsilon'')$ holds (see the Appendix).
\end{prop}

Before proving Proposition \ref{prop 1}, we will give a property of the operator $K$ definited in \eqref{K}:
\begin{lemma}\label{Kcomp}
The operator $K: L^2(\Omega)\longrightarrow H_{0, L}^1(\mathcal{C})$ is compact.
\end{lemma}
\begin{proof}
Let $(u_n)_n\subset L^2(\Omega)$ be a bounded sequence and set $v_n:=K(u_n)$ for all $n\in \N$. Since, by \eqref{K}, we have
\[
\|v_n\|^2\leq \|u\|_{L^2(\Omega)}\|v\|_{L^2(\Omega)},
\]
we immediately get that $(v_n)_n$ is bounded. Thus, we can suppose that there exists $v\in \HL$ such that $v_n\rightharpoonup v$ in $H_{0, L}^1(\mathcal{C})$ and $v_n\to v$ in $L^2(\Omega)$ as $n\to \infty$.
 In this way, since $$
 \int_{\mathcal{C}}{|Dv_n|^2\; dx dy}-\int_{\mathcal{C}}{Dv_n\cdot D v\; dx}-\int_{\Omega}{u_n(v_n-v)\; dx}=0$$
for all $n\in\mathbb N$, exploiting the previous convergences and recalling that $(u_n)_n$ is bounded in $L^2(\Omega)$, we immediately have that
$$
\|v_n\|^2\rightarrow \|v\|^2 \mbox{ as }n\to \infty,$$
that is, $v_n\rightarrow v$ in $H_{0, L}^1(\mathcal{C})$ as $n\to \infty$.
\end{proof}

\begin{proof}[Proof of Proposition $\ref{prop 1}$]
Assume by contradiction that there exists $\delta>0$ such that for all  $\epsilon_0>0$ there exist $\lambda\in [\lambda_{i-1}+\delta, \lambda_{j+1}-\delta]$ and $\epsilon', \epsilon'' \in (0,\epsilon_0)$, such that the condition $(\nabla)(f_\lambda,H_{i-1}\oplus H_j^\bot, \epsilon',\epsilon'')$ does not hold.

Take $\epsilon_0>0$ as given by Lemma \ref{Lemma 2}, and take a sequence $(u_n)_n$ in $H_{0, L}^1(\mathcal{C})$ such that  $f_\lambda (u_n)\in [\epsilon', \epsilon'']$ for every $n\in \N$, $d(u_n, H_{i-1}\oplus H_j^\bot)\rightarrow 0$ and $Q\nabla f_\lambda (u_n)\rightarrow 0$ as $n\to \infty$, i.e. $(u_n)$ is a sequence which makes $(\nabla)(f_\lambda,H_{i-1}\oplus H_j^\bot, \epsilon',\epsilon'')$ false.

Then, by Lemma \ref{Lemma 3}, we get that $(u_n)_n$ is bounded and we can assume that $u_n\rightharpoonup u$ in $H_{0, L}^1(\mathcal{C})$ and $u_n\to u$ in $L^r(\Omega)$ for all $r\in \left[1,2N/(N-1)\right)$ as $n\to \infty$.

Now, 
\[
Q\nabla f_\lambda (u_n)=u_n-Pu_n-K(\lambda u_n+g(x,u_n)),
\]
and we know that $g(x,u_n)\rightarrow g(x,u)$ in $L^{p/(p-1)}(\Omega)$ as $n\to \infty$ by $(g_2)$. Moreover, $Q\nabla f_\lambda (u_n)\rightarrow 0$ and $Pu_n\to 0$ as $n\to \infty$. Thus,  by Lemma \ref{Kcomp}, we find that $u_n\rightarrow u$ in $H_{i-1}\oplus H_j^\bot$ and $u$ is a critical point of $f_\lambda$ on $H_{i-1}\oplus H_j^\bot$. Therefore, by Lemma \ref{Lemma 2}, we know that $u=0$, while $0<\epsilon'\le f_\lambda(u_n)$ for all $n\in \mathbb N$, which contradicts the fact that $f_\lambda$ is continuous. Hence, the claim follows.
\end{proof}

In order to apply the $\nabla$-Theorem (see the Appendix), at this point we have only to show that $\sup f_\lambda (\overline{B_{j}(R)})$ is small enough. In order to do that, we need the following Lemma.
\begin{lemma}\label{Lemma 4}
We have $$\lim_{\lambda \rightarrow \lambda_j}{\sup f_\lambda (H_j)}=0.$$
\end{lemma}
\begin{proof}
Assume by contradiction that there exist $\epsilon >0$ and sequences $\mu_n \rightarrow \lambda_j$, $(u_n)_n$ in $H_j$, such that $$\sup f_{\mu_n} (H_j)=f_{\mu_n}(u_n)\ge \epsilon \quad \forall \; n\in \mathbb N.$$
Note that $(u_n)_n$ is well defined, since $f_{\mu_n}$ attains a maximum in $H_j$ thanks to condition $(g_5)$.

If $(u_n)_n$ is bounded, we can assume that $u_n\rightarrow u$ in $H_j$. In this way, by continuity, we have $$\epsilon \le \dfrac{1}{2}\int_\mathcal{C} {|Du|^2\; dx dy}-\dfrac{\lambda_j}{2}\int_{\Omega}{u^2\; dx}-\int_{\Omega}{G(x,u)\; dx} \le 0$$
by \eqref{1} and $(g_5)$, and a contradiction arises.
So we can assume that $\|u_n\|\rightarrow \infty$. 

In this case, by $(g_5)$ and the Sobolev inequality, there exists $\gamma_p>0$ such that
\begin{equation}\label{infhi}
0<\epsilon \le f_{\mu_n}(u_n) \le \dfrac{1}{2} \|u_n\|^2-\dfrac{\mu_n}{2\lambda_j}\|u_n\|^2-c_1 \gamma_p \|u_n\|^p,
\end{equation}
and since all norms are equivalent in $H_j$, the right hand side of the last inequality would tend to $-\infty$, which is absurd again.
\end{proof}

In order to prove Theorem \ref{Theorem 2}, first we need some preliminary results. 

\begin{thm}\label{Teorema 2}
Suppose there exist integers $i, j \ge 2$ such that $\lambda_{i-1}<\lambda_{i}=\cdots=\lambda_{j}<\lambda_{j+1}$. Then, there exist $\delta_1>0$, $\epsilon',\epsilon''>0$ such that for every $\lambda\in (\lambda_j -\delta_1, \lambda_j)$ problem \eqref{P} has at least two non trivial solutions $u_1, u_2$ with $f_\lambda(u_i)\in [\epsilon',\epsilon'']$, $i=1,2$.
\end{thm}
\begin{proof}
Take $\delta'>0$ and find $\epsilon_0$ as in Proposition \ref{prop 1}. Fix $\epsilon'<\epsilon''<\epsilon_0$. Then, by Lemma \ref{Lemma 4}, there exists $\delta_1\leq \delta'$ such that, if $\lambda \in (\lambda_j -\delta_1, \lambda_j)$, we have $\sup f_\lambda (H_j)<\epsilon ''$ and by Proposition \ref{prop 1} the condition $(\nabla)-(f_\lambda, H_{i-1}\oplus H_j^\bot, \epsilon',\epsilon'')$ holds. Moreover, since $\lambda<\lambda_j$, the topological structure of Lemma  \ref{Lemma 1} is satisfied. 

By the $\nabla$-Theorem (see Theorem \ref{Teorema 3} in the Appendix), there exist two critical points $u_1, u_2$ of $f_\lambda$ such that $f_\lambda (u_i)\in [\epsilon', \epsilon''], i=1,2$. In particular $u_1$ and $u_2$ are nontrivial solutions of  \eqref{P}, since $f_\lambda(0)=0$.
\end{proof}

In order to prove the existence of a third nontrivial solution, let us prove the following Lemma.
\begin{lemma}\label{Lemma 5}
Suppose there exist integers $i, j \ge 2$ such that $\lambda_{i-1}<\lambda_{i}=\cdots=\lambda_{j}<\lambda_{j+1}$. Then, there exists $\delta_i>0$, $\rho_1>0$ and $R_1>\rho_1$ such that for all $\lambda\in (\lambda_i -\delta_i, \lambda_i)$, we have
$$\inf f_\lambda (S_j^{+} (\rho_1))>\sup f_\lambda (T_{i, j+1} (R_1)).$$
In particular there exists a critical point $u_3$ of $f_\lambda$ such that $$f_\lambda (u_3)\ge \inf f_\lambda (S_j^{+} (\rho_1)).$$
\end{lemma}

\begin{proof}
Take $\lambda\in [\lambda_{i-1}, \lambda_i)$. By \eqref{2} and $(g_5)$ we get that for all $\tau>0$ there exist $\rho_1>0$ such that, if $v\in H_j^\bot$ and $\|v\|=\rho_1$, then 
\begin{equation}\label{11}
f_\lambda (v)\ge \dfrac{1}{2} \left(1-\dfrac{\lambda}{\lambda_{j+1}}-\tau \right)\rho_1^2.
\end{equation}
Choosing $\tau$ small enough, we get that $C=1-\dfrac{\lambda_j}{\lambda_{j+1}}-\tau>0$, so that \eqref{11} reads
\[
f_\lambda (v)\ge C\rho_1^2.
\]
By $(g_5)$ and  \eqref{1}, as in \eqref{infhi}, we get that $f_\lambda (u)\rightarrow -\infty$ if $u\in H_{j+1}$ and $\|u\|\rightarrow \infty$, since all the norms in $H_{j+1}$ are equivalent. By Lemma \ref{Lemma 4}, there exists $\delta_i>0$ such that $\forall \; \lambda$ in $(\lambda_i-\delta_i, \lambda_i)$, we have
$$
\sup f_\lambda (H_j)<C\rho_1^2.
$$
Finally, we choose $\delta_i<\delta_1$, where $\delta_1$ is the one found in Theorem \ref{Teorema 2}.

In this way, the classical Linking Theorem (see the Appendix) shows the existence of a critical point $u_3$ of $f_\lambda$ such that $f_\lambda(u_3)\ge C\rho_1^2$.
\end{proof}

Note that, although the topological structure found in Lemma \ref{Lemma 5} is equal to the one of Lemma \ref{Lemma 1}, it is not possible to apply the $\nabla$-theorem again, since it is not clear if $(\nabla)-\Big(f_\lambda,H_j \oplus H_{j+1}^\bot, C\rho_1^2, \sup f_\lambda (B_{j+1} (R_1))\Big)$ holds.

Now, we are ready to prove Theorem \ref{Theorem 2}, which becomes an obvious corollary of all the previous results.
\begin{proof}[Proof of Theorem $\ref{Theorem 2}$]
Take $\delta_i$ as given in Lemma \ref{Lemma 5}. Then the critical point $u_3$ found there is different from the critical points $u_1$, $u_2$ found in Theorem \ref{Teorema 2}, since $$f_\lambda (u_i)\le \sup f_\lambda (H_j) <C\rho_1^2 \le f_\lambda (u_3),$$
$i=1,2$.
\end{proof}
 
 \section{Appendix}
Although it is a well-known result in critical point theory, we recall here the linking theorem (see \cite[Theorem 5.3]{R} or \cite[Theorem 1.1]{PR}), in order to state an a priori estimate on critical values which we will use to prove Theorem \ref{Theorem 2}.
 \begin{thm}[Linking Theorem]
Let X be a Banach space such that $X=X_1 \oplus X_2$ with dim $X_1<\infty$, and let $f:X\to \R$ be of class $C^1$. Assume that
\begin{description}
\item $i)$ $f(0)=0$;
\item $ii)$ there exist $\rho, \alpha >0$ such that $\inf f(S_\rho \cap X_2)\ge \alpha $ and there exist $R>\rho$, $e\in X_2$ with $\|e\|=1$ such that $\sup f(\Sigma_R)\le 0$, where $S_\rho$ denotes the sphere in $X$ of radius $\rho$, with
$$\Sigma_R=\partial_{X_1 \oplus span(e)} \Delta_R \quad \mbox{and}$$
$$\Delta_R=\{u+te : u\in X_1, t>0, \|u+te\|\le R \}.$$
\item $iii)$ $(PS)_\beta$ holds, where 
\[
\beta=\inf_{h\in\textit{H}}\sup_{u\in \Delta_R} f(h(u))
\]
and
\[
\textit{H}=\{h\in C(\Delta_R, X): h_{|\Sigma_R }=Id\}.
\]
\end{description}
Then, $\beta$ is a critical value for $f$, and
\[
\alpha\leq \beta \leq  \sup_{u\in \Delta_R} f(u).
\]
\end{thm}

\begin{defn}[see \cite{ms}]
Let $X$ be a Hilbert space, $f:X\longrightarrow \mathbb R$ a $C^1$ function, $M$ a closed subspace of $X$ and $a,b\in \mathbb R \cup \{-\infty, +\infty \}$.

We say that the condition $(\nabla)- (f,M,a,b)$ holds if there exists $\gamma >0$ such that 
$$
\inf \; \Big\{ \|P_M \nabla f(u)\|: a\le f(u)\le b, d(u,M)\le \gamma \Big\}>0,
$$
where $P_M:X\longrightarrow M$ is the orthogonal projection of $X$ on $M$ and $d(u,M)=\displaystyle \inf_{z\in M}{d(u,z)}$ denotes the distance of $u$ from the subspace $M$.
\end{defn}
This means that, if the condition above holds, then $f_{|_M}$ cannot have critical points $u$ with $a\le f(u)\le b$ with some uniformity.
\begin{thm}[\textbf{$\nabla$-Theorem}, see \cite{ms}]\label{Teorema 3}
Let $X$ be a Hilbert space and let $X_i, i=1,2,3$ be three subspaces of $X$ such that $X=X_1\oplus X_2\oplus X_3$ with dim $X_i<\infty$ for $i=1,2$. Denote with $P_i:X\longrightarrow X_i$ the orthogonal projection of $X$ on $X_i$. Let $f: X\longrightarrow \mathbb R$ be a function of class $C^{1}$.
Let $\rho, \rho',\rho'',\rho_1$ be such that $\rho_1>0$, $0\le \rho'< \rho<\rho''$ and define
$$
\Delta=\Big\{u\in X_1 \oplus X_2 : \rho'\le \|P_2 u\|\le \rho'', \|P_1 u\|\le \rho_1 \Big\},
$$
$$
T=\partial_{X_1 \oplus X_2} \Delta, \quad S_{23}(\rho)=\Big\{u\in X_2\oplus X_3 : \|u\|=\rho\Big\}
$$ and
$$
B_{23} (\rho)=\Big\{u\in X_2 \oplus X_3 : \|u\|\le \rho \Big\}.
$$
Suppose that
$$
a':= \sup f(T)<\inf f(S_{23} (\rho))=:a''.
$$
Let $a$ and $b$ be such that $a'<a<a''$ and $b> \sup f(\Delta)$. Suppose that the condition $(\nabla)-(f,X_1\oplus X_3,a,b)$ holds and that $(PS)_c$ holds for all $c\in [a,b]$. Then, $f$ has at least two critical points in $f^{-1}([a,b])$. Moreover, if 
$$\inf f(B_{23} (\rho))>-\infty$$
and $(PS)_c$ holds for all $c \in [a_1,b]$ with 
\[
 a_1<\inf f(B_{23} (\rho)),
 \]
then $f$ has another critical level in $[a_1, a']$.
\end{thm}
Note that, in our context, the critical level in $[a_1,a']$ could vanish, so that no other nontrivial solution is obtained.

 \section*{Acknowledgement} The authors would like the thank the anonymous referee for her/his careful reading of the manuscript and her/his comments, which improved the final presentation.

\noindent Dimitri Mugnai\footnote{D. M. is member of the Gruppo Nazionale per
l'Analisi Matematica, la Probabilit\`a e le loro Applicazioni (GNAMPA)
of the Istituto Nazionale di Alta Matematica (INdAM), and is supported by the GNAMPA project {\sl Systems with irregular operators}}: Dipartimento di Matematica e Informatica, University of Perugia, Via Vanvitelli 1, 06123 Perugia - Italy\\
e-mail: dimitri.mugnai@unipg.it \medskip

\noindent
Dayana Pagliardini: Scuola Normale Superiore, Piazza dei Cavalieri 7, 56126 Pisa- Italy\\ e-mail:
dayana.pagliardini@sns.it\\

\end{document}